\documentclass[12pt, reqno]{amsart}
\usepackage{amsmath, amsthm, amscd, amsfonts, amssymb, graphicx, color}
\usepackage[bookmarksnumbered, colorlinks, plainpages]{hyperref}
\hypersetup{colorlinks=true,linkcolor=red, anchorcolor=green, citecolor=cyan, urlcolor=red, filecolor=magenta, pdftoolbar=true}
\newtheoremstyle{uprightplain}
{6pt}{6pt}
{\normalfont}
{}{\bfseries}{.}{ }{}
\theoremstyle{uprightplain}
\newtheorem{theorem}{Theorem}[section]
\newtheorem{lemma}[theorem]{Lemma}
\newtheorem{proposition}[theorem]{Proposition}
\newtheorem{cor}[theorem]{Corollary}
\theoremstyle{definition}

\newtheorem{question}[theorem]{Question}
\newtheorem{example}[theorem]{Example}
\newtheorem{remark}[theorem]{Remark}
\numberwithin{equation}{section}
\allowdisplaybreaks
\begin{document}
	\title [Numerical and Berezin range]{\Small{Numerical range and Berezin range of weighted composition operators on weighted Dirichlet spaces}}

	\author[S. Barik, A. Sen and K. Paul]{Somdatta Barik, Anirban Sen and Kallol paul}
	\address[Barik]{Department of Mathematics, Jadavpur University, Kolkata 700032, West Bengal, India}
	\email{bariksomdatta97@gmail.com}

	\address[Sen] {Mathematical Institute, Silesian University in Opava, Na Rybn\'{\i}\v{c}ku 1, 74601 Opava, Czech Republic}
	\email{anirbansenfulia@gmail.com; Anirban.Sen@math.slu.cz}

	\address[Paul] {Vice-Chancellor\\
		Kalyani University\\
		West Bengal 741235 \\and 
		Professor (on lien)\\ Department of mathematics\\ Jadavpur University\\Kolkata 700032\\West Bengal\\India}
	\email{kalloldada@gmail.com}
	
	
	
	\begin{abstract}  
		We investigate the numerical ranges of weighted composition operators on weighted Dirichlet spaces, focusing on the properties of the inducing functions. We identify conditions on these functions under which the origin lies in the interior of the numerical range. The geometric structure of the numerical range is also analyzed, determining when it contains a circular or elliptical disc and computing the corresponding radius. Next, we introduce a class of Weyl-type weighted composition operators and obtain their Berezin range and Berezin number. Finally, we characterize the convexity of the Berezin range for weighted composition operators on these spaces.
	\end{abstract}
	\maketitle	
	\tableofcontents

	\section{Introduction and preliminaries}
	
	Let $\mathbb D=\{z\in\mathbb C: |z|<1\}$ be the open unit disc and $\mathbb{T}=\{z \in \mathbb{C} : |z|=1\}$ be the unit circle of $\mathbb{C}.$ Let $\mathcal H(\mathbb D)$ denote the space of all holomorphic functions on $\mathbb D.$ We use the notation $\text{int}(X)$ for the interior of a set $X\in\mathbb C.$ For $s \in \mathbb{R},$ the weighted Dirichlet spaces $\mathcal{D}_{s}=\mathcal{D}_{s}(\mathbb{D})$ are defined as
	\[\mathcal{D}_{s}=\left\{f \in \mathcal H(\mathbb{D}) : f(z)=\sum_{n=0}^{\infty}a_nz^n, \sum_{n=0}^{\infty}(n+1)^{1-s}|a_n|^2< \infty \right\}.\]
	Weighted Dirichlet spaces are weighted Hardy spaces $H^2(\beta),$ where the associated weight sequence is $\beta(n)=(n+1)^{\frac{1-s}{2}}.$ This framework recovers several classical spaces: $\mathcal{D}_{1}$ is the classical Hardy-Hilbert space $H^2(\mathbb{D}),$ $\mathcal{D}_{2}$ is the classical Bergman space $A^2(\mathbb{D})$ and $\mathcal{D}_{0}$ coincides with the Dirichlet space $\mathcal{D}.$ In this article, we are primarily concerned with the case $0<s<1$ and we refer to the corresponding spaces $\mathcal D_s$
	as weighted Dirichlet spaces. 
	Each $\mathcal{D}_{s}$ is a separable Hilbert space with inner product 
	\[\langle f,g \rangle_{\mathcal{D}_{s}}=\sum_{n=0}^{\infty}\frac{\Gamma(n+1)\Gamma(s)}{\Gamma(n+s)}a_n \overline{b}_n,\]
	where $f(z)=\sum\limits_{n=0}^{\infty}a_nz^n$ and $g(z)=\sum\limits_{n=0}^{\infty}b_nz^n.$ The space $\mathcal D_s$ has an orthonormal basis $\{e_n\}_{n=0}^{\infty}$ given by $e_n(z)=\sqrt{\frac{\Gamma(n+s)}{\Gamma(n+1)\Gamma(s)}} z^n.$
	Moreover, these spaces are reproducing kernel Hilbert spaces with a complete Nevanlinna-Pick kernel. In such a case we will also say that the space is a complete Nevanlinna-Pick space (see \cite{AM_BOOK}). The reproducing kernel at $w \in \mathbb D$ is given by
	\begin{align}\label{r_1}
		k^{s}_w(z)= \frac{1}{(1-\bar{w}z)^s},~z \in \mathbb D.
	\end{align}
	The normalized reproducing kernel at $w \in \mathbb D$ is given by
	\begin{align}\label{r_2}
		\hat k^{s}_w(z)= \frac{(1-|w|^2)^\frac{s}{2}}{(1-\bar{w}z)^s},~z \in \mathbb D.
	\end{align}
	For further information on the weighted Dirichlet space, see \cite{Cowen_BOOK, MS_CJM_1986}.
	
	Let $\mathcal B(\mathcal D_s)$ denote the space of all bounded linear operators on $\mathcal D_s.$ The numerical range of $T\in\mathcal B(\mathcal D_s)$ is denoted by $W(T; \mathcal D_s)$ and defined as $$W(T; \mathcal D_s)=\{\langle Tf, f\rangle_{\mathcal D_s}:f\in \mathcal D_s, \|f\|_{\mathcal D_s}=1\}.$$ The numerical radius of $T$ on $\mathcal D_s$ is denoted by $w(T;  \mathcal D_s)$ and defined as 
	$$w(T; \mathcal D_s)=\sup\{|\langle Tf, f\rangle_{\mathcal D_s}|:f\in \mathcal D_s, \|f\|_{\mathcal D_s}=1\}. $$
	Note that, $W(T; \mathcal D_s)$ is a bounded convex subset of $\mathbb C.$ For further details on numerical ranges and numerical radius we refer to \cite{wu_gau_book}.
	%Furthermore, in the two-dimensional case, $W(T; \mathcal D_s)$ has a particularly nice geometric description: it is an ellipse whose foci are the eigenvalues of $T$. {\color{red}In higher dimensions, however, determining the shape of the numerical range becomes significantly more challenging and has been investigated by many authors.}\\ 
	
	For $T \in \mathcal B(\mathcal D_s),$ the Berezin transform of $T$ (\cite{BER1974,BER1972}) is the function $\widetilde{T}: \mathbb{D} \to \mathbb{C}$ defined by
	\[\widetilde{T}(z)=\langle T\hat{k}^{s}_z,\hat{k}^{s}_z \rangle_{\mathcal D_s},\,\,\,\,\ z \in \mathbb{D}.\]
	
	It is well known that the Berezin transform is a one-to-one, bounded, and real-analytic function on $\mathbb D,$ see \cite{BER1972}. Several properties of an operator are governed by its Berezin transform. In particular, the Berezin transform uniquely determines the operator, and the invertibility of the operator is likewise determined by its Berezin transform, see \cite{KAR_JFA_2006, ZZ_JOT_2016}. One can refer to \cite{C_PAMS_2012, E_JFA_1994, Zhu_AMS_2021} for the various aspects of the theory of Berezin transform.

	The Berezin range and Berezin radius of $T$ (\cite{K_CAOT_2013, KAR_JFA_2006}), denoted by $\textbf{Ber}(T)$ and  $\textbf{ber}(T),$ are respectively, defined as 
	\[\textbf{Ber}(T;\mathcal D_s)=\{\widetilde{T}(z): z \in \mathbb{D}\} \,\, 
	\text{and}\,\, \,\, \textbf{ber}(T;\mathcal D_s)=\sup_{z \in \mathbb{D}}|\widetilde{T}(z)|.\]
	From the definition, it follows immediately that the Berezin range is contained in the numerical range of the operator.

	For $\psi \in \mathcal H(\mathbb{D})$ and a holomorphic self-map $\phi$ of $\mathbb D,$ the weighted composition operator $C_{\psi,\phi }$ on  $\mathcal H(\mathbb{D})$ is defined by
	\[C_{\psi,\phi }f=\psi f\circ \phi,~~f \in \mathcal H(\mathbb{D}).\]
	In particular, when $\psi\equiv 1,$ $C_{\psi,\phi }$ becomes the composition operator $C_{\phi}.$ Notably, the boundedness and compactness of weighted composition operators on weighted Dirichlet spaces are studied in \cite{Sen1,Z_PAMS_1998}. In this article, our analysis is confined to weighted composition operators on $\mathcal D_s$ for $s \in (0,1).$ 
	
	With this notation in hand, we now give a broad overview of our results.
	%{\color{red}In this article, we are concerned with the numerical range and the Berezin range of weighted composition operators on $\mathcal D_s.$} 
	Motivated by \cite{BS_IEOT_2002,GJS_JMAA_2014} and the recent work in \cite{SHBP_CMB_25}, we initiate the study of the numerical range of weighted composition operators $C_{\psi,\phi}$ on $\mathcal D_s,$ focusing in particular on the question of whether the origin belongs to $W(C_{\psi,\phi};\mathcal D_s).$ 
	Although $0\notin W(C_{\psi,\phi};\mathcal D_s)$ in general, we show that this behavior changes under suitable assumptions on the symbols $\phi$ and $\psi.$ The following theorem presents our first result. 
	\begin{theorem}\label{Th_s1}
		Suppose $\phi$ is a holomorphic self-map of $\mathbb{D}$ and $\psi\in\mathcal{H}(\mathbb{D})$ with $C_{\psi,\phi}\in\mathcal{B}(\mathcal{D}_{s})$.\\
		$(i)$~ If $\phi$ is identity map on $\mathbb D$ and $\psi$ has a zero in $\mathbb{D}$ then $0\in W(C_{\psi,\phi}; \mathcal D_s).$\\
		$(ii)$~If $\phi$ is identity map on $\mathbb D$ and $\psi$ is continuous and vanishes at some point $z_0\in\mathbb T$ then $0\in \overline{W(C_{\psi,\phi}; \mathcal D_s)}.$\\
		$(iii)$~ If $\phi$ is non-identity on $\mathbb D$ then $0\in \overline{W(C_{\psi,\phi}; \mathcal D_s)}.$
		
	\end{theorem}
	Next, we address the problem of zero inclusion for weighted composition operators $C_{\psi,\phi} $ when the symbol $\phi$ fixes the origin but is neither a rotation nor a dilation. The following theorem provides a complete answer in this setting.
	\begin{theorem}\label{Th_s2}
		Suppose that $C_{\psi,\phi}\in\mathcal B(\mathcal D_s)$ and $\phi(0)=0.$ Unless $\phi$ is of the special form $\phi(z)=\lambda z$ with $\lambda\in\overline {\mathbb D},$ the origin belongs to the interior of the numerical range of $C_{\psi,\phi}$ on $\mathcal D_s$.
	\end{theorem}
	In contrast, when $\phi(z)=\lambda z,~~\lambda\in[-1, 0]$ then for certain choices of $\psi$, $0\in W(C_{\psi,\phi}; \mathcal D_s).$ This situation is addressed below.
	\begin{theorem}\label{Th_s3}
		Suppose that $C_{\psi,\phi}\in\mathcal B(\mathcal D_s)$ and $\psi$ is nonconstant. If $\phi(z)=\lambda z,~~\lambda\in[-1, 0],$ then the origin belongs to the interior of the numerical range of $C_{\psi,\phi}$ on $\mathcal D_s$.
	\end{theorem}
	%{\color{red} Before proceeding, we make several remarks concerning the preceding discussion.}
	
	We now move to the next stage of our work, focusing on the geometric characterization of the numerical range of weighted composition operators on $\mathcal D_s$. Understanding whether the numerical range contains a circular or elliptical disc is particularly useful, as it leads to explicit numerical radius bounds. In this direction, we determine the radius of the disc or the lengths of the major and minor axes. Our first result is the following.
	\begin{theorem}\label{Th_ss1}
		Let $C_{\psi,\phi}\in\mathcal B(\mathcal D_s)$ with $\phi(0)=0$ and $\psi$ vanishes at the origin to positive order $r.$ If $b_r$ denotes the $r$-th Taylor coefficient of $\psi$ then $W(C_{\psi,\phi}; \mathcal D_s)$ contains the disc centered at origin with radius $\frac{\Gamma(r+1)\Gamma(s)}{\Gamma(r+s)+\Gamma(r+1)\Gamma(s)} |b_r|.$
	\end{theorem}
	For specific choices of $\phi$ and under appropriate additional conditions, we demonstrate that the numerical range $W(C_{\psi,\phi}; \mathcal D_s)$ contains a disc. The corresponding results are given below:
	\begin{theorem}\label{Th_ss2}
		Suppose $C_{\psi,\phi}\in\mathcal B(\mathcal D_s)$ with $\phi(z)=\mu z,~\mu\neq0$ and $\psi(z)=\sum\limits_{n=1}^\infty b_nz^n.$ Then for each integer $r\geq2,$ $W(C_{\psi,\phi}; \mathcal D_s)$  contains the disc centered at origin with radius $\frac 12 \sqrt{\frac{\Gamma(r+1)\Gamma(s+1)}{\Gamma(r+s)}}|\mu b_{r-1}|.$
	\end{theorem}
	
	\begin{theorem}\label{Th_ss3}
		Suppose $C_{\psi,\phi}\in\mathcal B(\mathcal D_s)$ where $\phi(z)=\mu z,$ with $\mu= e^{i \frac {2\pi}{m}}$ and $\psi(z)=\sum\limits_{n=0}^\infty b_nz^n.$ Let $r_1, r_2\in\mathbb Z^+$ satisfy $r_1<r_2.$ If $b_{m r_1}b_{mr_2}b_{m(r_1-r_2)}=0$ while at least one of the coefficients $b_{m r_1},b_{mr_2}, b_{m(r_1-r_2)} $ is nonzero then $W(C_{\psi,\phi}; \mathcal D_s)$  contains the disc centered at $b_0$ with radius  
		\small{$$\frac12 \sqrt{\frac{\Gamma(mr_1+1)\Gamma(s)}{\Gamma(mr_1+s)}|b_{mr_1}|^2+\frac{\Gamma(mr_2+1)\Gamma(s)}{\Gamma(mr_2+s)}|b_{mr_2}|^2+\frac{\Gamma(mr_1+s)\Gamma(mr_2+1)}{\Gamma(mr_1+1)\Gamma(mr_2+s)}|b_{m(r_2-r_1)}|^2}.$$}
	\end{theorem}

	Moreover, we show that under appropriate assumptions the numerical range $W(C_{\psi,\phi}; \mathcal D_s)$ contains an elliptical region. The detailed statement of the results are given below.
	\begin{theorem}\label{Th_ss4}
		Suppose $C_{\psi,\phi}\in\mathcal B(\mathcal D_s)$ where $\phi(z)=\mu z,$ with $\mu= e^{i \frac {2\pi}{m}}$ and $\psi(z)=\sum\limits_{n=0}^\infty b_nz^n$ with
		$b_{mr+k}\neq 0$ for some $k\in(0,m).$ Then $W(C_{\psi,\phi}; \mathcal D_s)$  contains the ellipse with foci at the points $b_0$, $b_0e^{i \frac {2\pi(mr+k)}{m}}$ with major axis of length $\sqrt{\frac{\Gamma(mr+k+1)\Gamma(s)}{\Gamma(mr+k+s)}|b_{mr+k}|^2+|b_0|^2|1-e^{i \frac {2\pi(mr+k)}{m}}|^2}$ and minor axis of length $$\sqrt{\frac{\Gamma(mr+k+1)\Gamma(s)}{\Gamma(mr+k+s)}}|b_{mr+k}|.$$ 
	\end{theorem}
	\begin{theorem}\label{Th_ss5}
		Suppose $C_{\psi,\phi}\in\mathcal B(\mathcal D_s)$ where $\phi(z)=\mu z,$ with $\mu= e^{i2\pi \phi},$ $\phi$ is irrational and $\psi(z)=\sum\limits_{n=0}^\infty b_nz^n.$ Let $p\geq0$ and $q>0.$ Then  $W(C_{\psi,\phi}; \mathcal D_s)$  contains the ellipse with foci at the points $ b_0e^{i2\pi p\phi}$, $b_0e^{i 2\pi(p+q)\phi}$ with major axis of length $$\sqrt{|b_0|^2|e^{i2\pi p\phi}-e^{i 2\pi(p+q)\phi}|^2+\frac{\Gamma(p+q+1)\Gamma(p+s)}{\Gamma(p+q+s)\Gamma(p+1)}|b_q|^2}$$ and minor axis of length $\sqrt{\frac{\Gamma(p+q+1)\Gamma(p+s)}{\Gamma(p+q+s)\Gamma(p+1)}}|b_{q}|$.
	\end{theorem}
	
	We proceed to the next part of our analysis by introducing a special class of weighted composition operators on $\mathcal D_s,$ denoted by $C_{\hat k_\gamma^{s}, \phi_{\gamma, \alpha}}$, called Weyl-type weighted composition operators. Earlier, we examined the inclusion of the origin in the numerical range. Since the Berezin range is always contained in the numerical range, this naturally leads to the question of whether the origin lies in the Berezin range. Our investigation for the Weyl-type operator yields the following result.
	\begin{theorem}\label{Th_sss_1}
		For the Weyl-type weighted composition operators $C_{\hat k_\gamma^{s}, \phi_{\gamma, \alpha}}$ on $\mathcal D_s,$ 
		$$0\in\overline{\textbf{Ber}(C_{\hat k_\gamma^{s}, \phi_{\gamma, \alpha}}; \mathcal D_s)}\setminus\textbf{Ber}(C_{\hat k_\gamma^{s}, \phi_{\gamma, \alpha}}; \mathcal D_s),$$ whenever $(\alpha, \gamma)\neq(1,0).$
	\end{theorem}
	
	For suitable choices of $\alpha,$ we determine the exact Berezin range of $C_{\hat k_\gamma^{s}, \phi_{\gamma, \alpha}}$ on $\mathcal D_s,$ which is described as follows. 
	\begin{theorem}\label{Th_sss_2}
		For the Weyl-type weighted composition operators $C_{\hat k_\gamma^{s}, \phi_{\gamma,-1}}$ on $\mathcal D_s,$ $\textbf{Ber}\left(C_{\hat k_\gamma^{s}, \phi_{\gamma,-1}}; \mathcal D_s\right)=(0,1].$
	\end{theorem}
	
	Furthermore, for another choice of the parameter $\alpha,$ we obtain the Berezin radius of $(C_{\hat k_\gamma^{s}, \phi_{\gamma, \alpha}})$ on $\mathcal D_s.$ 
	
	\begin{theorem}\label{TT1}\label{Th_sss_3}
		For the Weyl-type weighted composition operators $C_{\hat k_\gamma^{s}, \phi_{\gamma, 1}}$ on $\mathcal D_s,~~~~\textbf{ber}\left(C_{\hat k_\gamma^{s}, \phi_{\gamma, 1}}; \mathcal D_s\right)=(1-|\gamma|^2)^{\frac{s}{2}}.$ 
	\end{theorem}
	Moreover, we compute the Berezin radius of $\mathcal X_\gamma=C_{\hat k_\gamma^{s}, \phi_{\gamma, 1}}+C_{\hat k^s_{-\gamma}, \phi_{-\gamma, 1}},$ and show that $\mathcal X_\gamma$ satisfies the reverse power inequality for the Berezin radius.

	The final part of this work is devoted to the study of convexity of the Berezin range. Convexity is a classical and fundamental property of the numerical range of bounded operators, which naturally motivates the question of whether an analogous property holds for the Berezin range. Recently, in \cite{CC_LAA_2022}, the authors characterized the convexity for a class of composition operators on Hardy-Hilbert space, and this was subsequently explored in Bergman and weighted Bergman spaces in \cite{AGS_CAOT_2023, SBP_CAOT_25}.
	Here, we further investigate this on the weighted Dirichlet spaces $\mathcal D_s$ and obtain characterizations in several significant cases.
	For $\phi_{0,\xi}(z)=\xi z,$ we have the following result.

	\begin{theorem}\label{Th_ssss1}
		Suppose $C_{{\phi_{0,\xi}}}\in\mathcal B(\mathcal D_s)$ with ${\phi_{0,\xi}}(z)=\xi z, \xi\in\overline{\mathbb D}$ and $z\in\mathbb D.$ Then $\textbf{Ber}(C_{{\phi_{0,\xi}}}; \mathcal D_s)$ is convex if and only if $\xi\in[-1, 1].$ 
	\end{theorem}

	We next consider the Blaschke factor $\phi_{\gamma, 1}=\frac{z-\gamma}{1-\bar{\gamma}z}$ where $\gamma, z \in\mathbb D,$ which leads to the following result.
	
	\begin{theorem}\label{Th_ssss2}
		Let $C_{\phi_{\gamma, 1}}\in \mathcal B(\mathcal D_s).$ Then $\textbf{Ber}(C_{\phi_{\gamma, 1}};\mathcal D_s)$ is convex if and only if $\gamma=0$.
	\end{theorem}
	
	\textbf{Organization of the paper.} 
	In Section \ref{sct_2}, we prove Theorem \ref{Th_s1}, Theorem \ref{Th_s2} and Theorem \ref{Th_s3} concerning the inclusion of zero in the numerical range of $C_{\psi,\phi}$ and we provide examples to illustrate their applicability. In Section \ref{sct_3}, we establish Theorem \ref{Th_ss1}, Theorem \ref{Th_ss2} and Theorem \ref{Th_ss3}, which show that the numerical range of $C_{\psi,\phi}$ contains a circular disc. We also prove Theorem \ref{Th_ss4} and Theorem \ref{Th_ss5} regarding the inclusion of an elliptical disc in the numerical range of $C_{\psi,\phi}$ on $\mathcal D_s,$
	along with related examples. In Section \ref{sct4}, we present proofs of Theorem \ref{Th_sss_1}, Theorem \ref{Th_sss_2} and Theorem \ref{Th_sss_3} on the Berezin range and Berezin radius of Weyl-type weighted composition operators on $\mathcal D_s$ and we obtain a class of operators satisfying the reverse power inequality for the Berezin number. The article concludes with Section \ref{sct5}, which contains the proof of Theorem \ref{Th_ssss1} and Theorem \ref{Th_ssss2}, along with related results on the convexity of the Berezin range of composition operators on $\mathcal D_s.$

	\section{Containment of zero in $W(C_{\psi,\phi}; \mathcal D_s)$}\label{sct_2}
	
	In this section, we address the inclusion of zero in the numerical range of 
	$C_{\psi,\phi}$ on $\mathcal D_s.$ To prove the Theorem~\ref{Th_s1}, we first recall the definition of a radial limit. A function $f\in\mathcal H(\mathbb D)$ is said to admit a radial limit if $\lim\limits_{r\to 1}f(r e^{i\theta})$ exists almost everywhere in $\mathbb T.$ It is shown in \cite[Th. 2.2]{Cowen_BOOK} that for each $f\in H^{2}(\mathbb D)$, there corresponds a function $f^\ast\in L^2(\mathbb T),$ defined almost everywhere by $f^\ast(e^{i\theta})=\lim\limits_{r\to 1}f(r e^{i\theta})$. Furthermore, if $f^\ast(e^{i\theta})=0$ for almost all $e^{i\theta}$ on some arc $I\subseteq\mathbb T$ then $f(z)=0$ for all $z\in\mathbb D.$ We now proceed to prove the first result of this section.
	
	\begin{proof}[Proof of Theorem \ref{Th_s1}]
		Since $C_{\psi,\phi}$ is bounded on $\mathcal D_s,$ a direct application of the reproducing property of the kernel functions yields 
		$C^\ast_{\psi,\phi} k_z^s=\overline{\psi(z)}k_{\phi(z)}^s.$ Thus,
		for any $z\in\mathbb D,$
		\begin{align}\label{eq_n1}
			\langle C_{\psi,\phi}\hat k_z^s, \hat k_z^s\rangle_{\mathcal D_s}
			&=\frac {1}{\|k_z^s\|_{\mathcal D_s}^2}\langle k_z^s, C_{\psi,\phi}^\ast k_z^s\rangle_{\mathcal D_s}\nonumber\\
			&= \frac {1}{\|k_z^s\|_{\mathcal D_s}^2}\langle k_z^s, \overline{\psi(z)}k_{\phi(z)}^s\rangle_{\mathcal D_s}\nonumber\\
			&=\psi(z)\left(\frac{1-|z|^2}{1-\bar z \phi(z)}\right)^s.
		\end{align}
		$(i)$~ Let $\phi$ be identity map on $\mathbb D.$ Then it follows from \eqref{eq_n1} that $\langle C_{\psi,\phi}\hat k_z^s, \hat k_z^s\rangle_{\mathcal D_s}=\psi(z).$ Thus, if there exists $z_0\in\mathbb D$ such that $\psi(z_0)=0$ then $\langle C_{\psi,\phi}\hat k_{z_{0}}^s, \hat k_{z_{0}}^s\rangle_{\mathcal D_s}=0.$ Therefore, $0\in W(C_{\psi,\phi}; \mathcal D_s).$ \\
		$(ii)$~If $\psi$ is continuous at $z_0\in\mathbb T$ with $\psi(z_0)=0,$ then $\lim\limits_{z\to z_0} \langle C_{\psi,\phi}\hat k_z^s, \hat k_z^s\rangle_{\mathcal D_s}=0$. This implies that $0\in \overline{W(C_{\psi,\phi}; \mathcal D_s)}.$\\
		$(iii)$~ Supopose $\phi$ is not identity map on $\mathbb D$. Then there exists $z_0\in \mathbb T$ such that $\phi^\ast(z_0)\neq z_0.$ From the closed graph theorem we observe that $\psi\in\mathcal D_s\subseteq H^2(\mathbb D).$ Thus, the radial limit $\psi^\ast(z_0)$ exists. From \eqref{eq_n1}, $\lim\limits_{z\to z_0} \langle C_{\psi,\phi}\hat k_z^s, \hat k_z^s\rangle_{\mathcal D_s}=0.$ Therefore, we obtain $0\in \overline{W(C_{\psi,\phi}; \mathcal D_s)}.$
	\end{proof}

	If we consider $\phi$ as a constant self-map of $\mathbb D$ such that $\phi\equiv v$ then for any $g\in\mathcal D_s$ $C_{\psi,\phi}(g)=\psi \langle g, k_v^s\rangle_{\mathcal D_s}.$ Hence $C_{\psi,\phi}$ is a rank one operator. By using \cite[Prop. 2.5]{BS_IEOT_2002}, we get the following proposition. 
	\begin{proposition}\label{prop_1}
		Suppose that $\phi$ is a constant self-map of $\mathbb D$ such that $\phi\equiv v.$\\ 
		$(i)$ If $k_v^s =t \psi$ for some $t\in\mathbb C\setminus \{0\}$ then $W(C_{\psi,\phi}; \mathcal D_s)$ is a closed line segment from $0$ to $\bar t \|\psi\|_{\mathcal D_s}^2.$\\
		$(ii)$ If $k_v^s\perp \psi$ then $W(C_{\psi,\phi}; \mathcal D_s)$ is the closed disc centered at origin, of radius $\frac{\|\psi\|_{\mathcal D_s}}{2 (1-|v|^2)^{s/2}}.$\\
		$(iii)$ Otherwise $W(C_{\psi,\phi}; \mathcal D_s)$ is a closed elliptical disc with foci at $0$ and $\psi(v).$
	\end{proposition}
	Our next result relies on the following lemma, which is a consequence of \cite[Th. 2.6 and Th. 2.9]{KEY_JME_21}.
	\begin{lemma}\label{lm_g1}
		Suppose that $C_{\psi,\phi}\in\mathcal B(\mathcal D_s)$, $\psi$ is non-zero and $\phi$ is a nonconstant self-map of $\mathbb D$. If $\phi$ is not univalent or $\psi$ has a zero on $\mathbb D$ then the origin belongs to the interior of the numerical range of $C_{\psi,\phi}$ on $\mathcal D_s$.
	\end{lemma}

	We are now in a position to prove the Theorem \ref{Th_s2}.
	\begin{proof}[Proof of Theorem \ref{Th_s2}]
		Assume $\phi'(0)\neq 0.$ For $\phi'(0)= 0,$ Lemma \ref{lm_g1} ensures the origin is an interior point of the numerical range of $C_{\psi,\phi}$ on $\mathcal D_s$. Let $\phi$ be not of the form $\phi(z)=\lambda z$ with $\lambda\in\overline {\mathbb D}.$ Then 
		$$\phi(z)=\phi'(0) z\left(1+a z^p\left(\sum\limits_{k=1}^{\infty}c_kz^k\right)\right),$$
		where $a\neq 0$ and $p\in \mathbb N$.
		Thus, for each $r\in\mathbb N,$
		$$(\phi(z))^r=(\phi'(0))^rz^r+ra(\phi'(0))^rz^{r+p}+\text{terms of higher order in}~ z.$$
		The matrix with respect to the orthonormal basis $\left\{\sqrt{\frac{\Gamma(r+s)}{\Gamma(r+1)\Gamma(s)}}z^r\right\}_{r=0}^\infty $
		has its $r$-th column the sequence of power series coefficients of $\sqrt{\frac{\Gamma(r+s)}{\Gamma(r+1)\Gamma(s)}}\psi \phi^r.$ Consider the two dimensional subspace 
		$$P_r=span\left\{\sqrt{\frac{\Gamma(r+s)}{\Gamma(r+1)\Gamma(s)}}z^r, \sqrt{\frac{\Gamma(r+p+s)}{\Gamma(r+p+1)\Gamma(s)}}z^{r+p}\right\}.$$
		%Clearly, $P_r$ is a two dimensional subspace of $\mathcal D_s.$ 
		Let $\psi(z)=\sum\limits_{n=0}^\infty b_n z^n.$ Then the matrix representation of $C_{\psi,\phi}$ on $P_r$ with respect to the basis $\Big\{\sqrt{\frac{\Gamma(r+s)}{\Gamma(r+1)\Gamma(s)}}z^r,\\
		\sqrt{\frac{\Gamma(r+p+s)}{\Gamma(r+p+1)\Gamma(s)}}z^{r+p}\Big\}$ is given by
		$$\begin{bmatrix}
			b_0(\phi'(0))^r&0\\
			\sqrt{\frac{\Gamma(r+p+1)\Gamma(r+s)}{\Gamma{(r+1)}\Gamma{(r+p+s)}}} (b_p+rab_0)(\phi'(0))^r&b_0(\phi'(0))^{r+p}
		\end{bmatrix}=(\phi'(0))^r M_r,$$ where 
		$$M_r=\begin{bmatrix}
			b_0&0\\
			\sqrt{\frac{\Gamma(r+p+1)\Gamma(r+s)}{\Gamma{(r+1)}\Gamma{(r+p+s)}}} (b_p+rab_0)&b_0(\phi'(0))^{p}
		\end{bmatrix}.$$ 
		In view of the inclusion property of numerical ranges for compressions, it remains to show that $0\in \text{int}~W(M_r; P_r)$ for some $r.$
		%Since the numerical range of compression is contained in the numerical range of the operator so it is sufficient to show that $0\in int(M_r).$
		For $b_0=0,$ Lemma \ref{lm_g1} implies that  $0\in \text{int}~ W(M_r; P_r)$. Let $b_0\neq0.$ Then $W(M_r; P_r)$ is an elliptical disc with foci at the points $b_0$ and $b_0(\phi'(0))^{p}$ and minor axis of length $\sqrt{\frac{\Gamma(r+p+1)\Gamma(r+s)}{\Gamma{(r+1)}\Gamma{(r+p+s)}}} |b_p+rab_0|.$ 
		By a simple computation, we obtain $$\lim\limits_{r\to\infty}\frac{\Gamma(r+p+1)\Gamma(r+s)}{\Gamma{(r+1)}\Gamma{(r+p+s)}}=1.$$
		By choosing $r$ sufficiently large, the length of the minor axis of $W(M_r; P_r)$ exceeds the modulus of its center. Thus there exists $r$ such that $0\in int~W(M_r; P_r).$
	\end{proof}
	
	In continuation of the above discussion, we prove Theorem \ref{Th_s3}.
	
	\begin{proof}[Proof of Theorem \ref{Th_s3}]
		Suppose $\psi(0)=0.$ Lemma \ref{lm_g1} yields that the origin belongs to the interior of the numerical range of $C_{\psi,\phi}$ on $\mathcal D_s$ for the case $\lambda\in [-1,0).$ The remaining case follows from Proposition \ref{prop_1}. Now let $\psi(0)\neq0.$ Without loss of generality, assume that $\psi(z)=1+\zeta(z),$ where $\zeta$ is a nonconstant analytic function satisfying $\zeta(0)=0.$ Clearly, $C_{\zeta,\phi}$ is bounded on $\mathcal D_s.$  Note that, 
		$$\langle C_{\psi,\phi}g,g \rangle_{\mathcal D_s}=\langle C_{\phi}g,g \rangle_{\mathcal D_s}+\langle C_{\zeta,\phi}g,g \rangle_{\mathcal D_s},~~g\in\mathcal D_s~~\text{with}~~\|g\|_{\mathcal D_s}=1~~\text{and}~~\langle C_{\phi}g,g \rangle_{\mathcal D_s} \in \mathbb R.$$ 
		Lemma \ref{lm_g1} implies that $W(C_{\zeta,\phi}; \mathcal D_s)$ contains a disc centered at the origin for $\lambda\in[-1,0)$. Therefore, there exists $h\in\mathcal D_s$ with $\|h\|_{\mathcal D_s}=1$ such that $Im~\langle C_{\zeta,\phi}h,h \rangle_{\mathcal D_s}>0.$ Consequently, $k_1=\langle C_{\psi,\phi}h,h \rangle_{\mathcal D_s}$ lies in the upper half plane. By a similar argument, we obtain another point $k_2$ in the lower half plane. Moreover, we have $\langle C_{\psi,\phi}e_1,e_1 \rangle_{\mathcal D_s}=\lambda$ and $\langle C_{\psi,\phi}e_0,e_0 \rangle_{\mathcal D_s}=1.$ Hence, we get
		$$0\in \text{int}~ (span\{1, \lambda, k_1, k_2\})~\subseteq \text{int}~W(C_{\psi,\phi}; \mathcal D_s).$$
		For 
		$\lambda=0,$ the conclusion follows directly from Proposition~\ref{prop_1}.
	\end{proof}
	It remains to investigate the case $\phi(z)=\lambda z,$ where $\lambda\in(0, 1]$. In this setting, a definitive conclusion is not possible, as zero may or may not lie in the interior of the numerical range of $C_{\psi,\phi}$. The different cases are discussed below.\\
	$(i)$~ If $\psi$ is constant and $\phi(z)=\lambda z,~\lambda\in[-1,1]$ then $W(C_{\psi,\phi};\mathcal D_s)$ is a line segment of $\mathbb C$ and hence $0$ does not belong to the interior of $ W(C_{\psi,\phi};\mathcal D_s).$\\
	$(ii)$~ If $\psi$ is nonconstant with $\psi(0)=0$ and $\phi(z)=\lambda z,~\lambda\in(0,1)$ then it follows from Lemma \ref{lm_g1} that $0$ belongs to the interior of $ W(C_{\psi,\phi};\mathcal D_s).$\\
	$(iii)$~ If $\psi$ is nonconstant with $\psi(0)\neq0$ and $\phi(z)=\lambda z,~\lambda\in(0,1)$ then the following two cases arise:\\
	$Case~1:$~ Suppose $\psi(0)\neq0$ and $\psi(z_0)=0$ for some $z_0\in\mathbb D.$ It follows from Lemma \ref{lm_g1} that $0$ belongs to the interior of $ W(C_{\psi,\phi};\mathcal D_s).$ \\
	$Case~2:$~It remains to consider the case where $\psi(z)\neq 0$ for every $z\in\mathbb D$. We next provide an example showing that the desired conclusion is not possible in general under this assumption.
	\begin{example}
		Consider $\psi(z)=1+\frac z8$ and $\phi(z)=\frac z4.$ Let $g\in\mathcal D_s$ with $\|g\|_{\mathcal D_s}=1$ and $g(z)=\sum\limits_{n=0}^\infty a_n z^n.$ Then 
		$\langle C_{\psi,\phi} g,g\rangle_{\mathcal D_s}=\eta+\frac 18 \zeta,$ where $$\eta=\sum\limits_{n=0}^{\infty}\frac{\Gamma(n+1)\Gamma(s)}{4^n\Gamma(n+s)}|a_n |^2~~ \text{and} ~~\zeta=\sum\limits_{n=0}^{\infty}\frac{\Gamma(n+2)\Gamma(s)}{4^n\Gamma(n+s+1)}a_n \overline{a_{n+1}}.$$ As $\eta>0,$ it easy to observe that $|\zeta|\leq \frac 52 \eta.$ Thus, $\eta+\frac 18 \zeta\neq 0$ and so $0\notin W(C_{\psi,\phi};\mathcal D_s).$ 
	\end{example}

	It is well known that if $0$ belongs to the numerical range of a compact operator, then its numerical range is closed. Combining this observation with the preceding results, we establish sufficient conditions for the closedness of compact weighted composition operators on $\mathcal D_s.$
	\begin{cor}\label{cor_s1}
		Suppose $C_{\psi,\phi}$ is compact on $\mathcal D_s.$ Then $W(C_{\psi,\phi}; \mathcal D_s)$ is closed whenever one of the following conditions is satisfied:\\
		$(i)$~$\phi(0)=0$ and $\phi$ is not of the form $\phi(z)=\lambda z$ with $\lambda\in\overline {\mathbb D}.$ \\
		$(ii)$~$\psi$ is nonconstant and $\phi(z)=\lambda z,~~\lambda\in(-1, 0].$\\
		$(iii)$~$\psi$ is nonzero and either $\psi$ has a zero in $\mathbb D$ or $\phi$ is not univalent.
	\end{cor}
	We now provide some examples to demonstrate the applicability of the above results.
	\begin{example}
		$(i)$~Let $\phi(z)=\frac{3z}{4-z}$ and $\psi(z)=e^z$. Then from \cite[Th. 1.2]{Sen1}, it is easy to observe that $C_{\psi,\phi}$ is compact on $\mathcal D_s.$ Thus, by Theorem \ref{Th_s2}, $0$ lies in the interior of $W(C_{\psi,\phi};\mathcal D_s)$  and Corollary \ref{cor_s1} guarantees that $W(C_{\psi,\phi};\mathcal D_s)$ is closed.\\
		$(ii)$ Let $\phi(z)=\lambda z,~\lambda\in(-1, 0]$ and $\psi(z)=\frac{az+b}{cz+d},$ where $a,b,c,d\in\mathbb C$ with $ad-bc\neq0,$ is a self-map of $\mathbb D.$ Then from \cite[Th. 1.2]{Sen1}, we have $C_{\psi,\phi}$ is compact on $\mathcal D_s.$ Consequently, Theorem \ref{Th_s2} implies that $0$ is an interior point of  $W(C_{\psi,\phi};\mathcal D_s),$ while Corollary \ref{cor_s1} ensures the closedness of this set.\\
		$(iii)$ Let $\phi(z)=\lambda z,~\lambda\in(-1, 0]$ and $\psi(z)=e^{z-1}.$ Then from \cite[Th. 1.2]{Sen1}, $C_{\psi,\phi}$ is compact on $\mathcal D_s.$  Thus, it follows from Theorem \ref{Th_s3} that $0$ belongs to the interior of  $W(C_{\psi,\phi};\mathcal D_s)$ and Corollary \ref{cor_s1} ensures that $W(C_{\psi,\phi};\mathcal D_s)$ is closed.
	\end{example}

	\section{Containment of circle or ellipse in $W(C_{\psi,\phi}; \mathcal D_s)$}\label{sct_3}
	
	This section is devoted to studying the geometric characterization of the numerical range of weighted composition operators on $\mathcal D_s.$ We begin by proving Theorem~\ref{Th_ss1}, which ensures the inclusion of a circular disc in $W(C_{\psi,\phi}; \mathcal D_s)$.
	\begin{proof}[Proof of Theorem \ref{Th_ss1}]
		Suppose $g(z)=\frac{1}{\sqrt{1+\frac{\Gamma(r+1)\Gamma(s)}{\Gamma(r+s)}}}(\mu+z^r)~~\forall z\in\mathbb D$ where $\mu\in\mathbb T.$ Then $g\in\mathcal D_s$ and $\|g\|_{\mathcal D_s}=1.$ Let $\phi(z)=\sum\limits_{n=1}^{\infty}a_nz^n$ and $\psi(z)=\sum\limits_{n=r}^{\infty}b_nz^n.$ Then \begin{eqnarray*}
			C_{\psi,\phi}(g)&&=\psi(z) g(\phi(z))\\
			&&=\frac{1}{\sqrt{1+\frac{\Gamma(r+1)\Gamma(s)}{\Gamma(r+s)}}}\left(\sum_{n=r}^{\infty}b_nz^n \right)\left(\mu+\left(\sum_{n=1}^{\infty}a_nz^n\right)^r\right)\\
			&&=\frac{1}{\sqrt{1+\frac{\Gamma(r+1)\Gamma(s)}{\Gamma(r+s)}}}\left(\sum_{n=r}^{\infty}b_nz^n \right)\left(\mu+a_1^rz^r+\text{higher order terms in}~z\right).
		\end{eqnarray*} 
		Hence,
		\begin{eqnarray*}
			&&\langle C_{\psi,\phi}g, g\rangle_{\mathcal D_s}\\
			&&=\frac{1}{{1+\frac{\Gamma(r+1)\Gamma(s)}{\Gamma(r+s)}}}\left\langle \left(\sum_{n=r}^{\infty}b_nz^n \right)\left(\mu+a_1^rz^r+\text{higher order terms in}~z\right), \mu+z^r\right\rangle_{\mathcal D_s}\\
			&&=\frac{\Gamma(r+1)\Gamma(s)}{\Gamma(r+s)+\Gamma(r+1)\Gamma(s)}\mu b_r.
		\end{eqnarray*}
		As $\mu\in\mathbb T$ is arbitrary, $W(C_{\psi,\phi}, \mathcal D_s)$ contains the disc centered at origin with radius $\frac{\Gamma(r+1)\Gamma(s)}{\Gamma(r+s)+\Gamma(r+1)\Gamma(s)} |b_r|.$        
	\end{proof}

	Next, we provide the proof of Theorem~\ref{Th_ss2} and Theorem ~\ref{Th_ss3}, which ensures, for particular choices of $\phi$ and under suitable additional assumptions, $W(C_{\psi,\phi}; \mathcal D_s)$ contains a circular disc.
	\begin{proof}[Proof of Theorem \ref{Th_ss2}]
		Let $P_r=\text {span} \{e_1, e_r\},~r\geq2.$ Now,  $$C_{\psi,\phi}e_1(z)=\mu\sqrt s \sum\limits_{n=1}^\infty b_nz^{n+1} $$ and $$C_{\psi,\phi}e_r(z)=\mu^r\sqrt {\frac{\Gamma(r+s)}{\Gamma(r+1)\Gamma(s)}} \sum\limits_{n=1}^\infty b_nz^{n+r}.$$ Hence
		$\begin{bmatrix}
			0&0\\
			\sqrt{\frac{\Gamma(r+1)\Gamma(s+1)}{\Gamma(r+s)}}\mu b_{r-1}&0
		\end{bmatrix}$ is the matrix representation of the compression of $C_{\psi,\phi}$ to $P_r.$ Thus, the numerical range of the compression of $C_{\psi,\phi}$ to $P_r$ is a closed disc centered at origin with radius \small{$\frac 12 \sqrt{\frac{\Gamma(r+1)\Gamma(s+1)}{\Gamma(r+s)}}|\mu b_{r-1}|.$} Since the numerical range of compression is contained in the numerical range of the operator so $W(C_{\psi,\phi}; \mathcal D_s)$  contains the disc centered at origin with radius $\frac 12 \sqrt{\frac{\Gamma(r+1)\Gamma(s+1)}{\Gamma(r+s)}}|\mu b_{r-1}|.$
	\end{proof}
	We now present an example to illustrate the applicability of the above theorem.
	\begin{example}
		Suppose $\psi(z)=\frac{\alpha z}{1-\beta z},~\alpha\neq0,$ is a holomorphic self-map of $\mathbb D$ and $\phi(z)=\mu z,~|\mu|\leq1, \mu\neq0.$ Then $|\beta|<1$ and $|\alpha|\leq 1-|\beta|.$ Form \cite[Th. 1.1]{Sen1}, we have $C_{\psi,\phi}\in\mathcal B(\mathcal D_s).$ It follows from Theorem \ref{Th_ss2} that for each integer $r\geq2,$ $W(C_{\psi,\phi}; \mathcal D_s)$  contains the disc centered at origin with radius $\frac 12 \sqrt{\frac{\Gamma(r+1)\Gamma(s+1)}{\Gamma(r+s)}}|\mu \alpha\beta^{r-2}|.$ Now the function $f(r)=\frac 12 \sqrt{\frac{\Gamma(r+1)\Gamma(s+1)}{\Gamma(r+s)}}|\mu \alpha\beta^{r-2}|$ is decreasing for $r\geq2.$ Thus, $w(C_{\psi,\phi}; \mathcal D_s)\geq f(2)=\frac{1}{\sqrt{2(s+1)}}|\mu\alpha|.$ 
	\end{example}

	\begin{proof}[Proof of Theorem \ref{Th_ss3}]
		Suppose $P=\text{span}~\{e_0, e_{mr_1}, e_{mr_2}\}.$ Then $$C_{\psi,\phi} e_0(z)=\sum_{n=0}^\infty b_nz^n,$$ $$C_{\psi,\phi} e_{mr_1}(z)=\sqrt{\frac{\Gamma(mr_1+s)}{\Gamma(mr_1+1)\Gamma(s)}}\sum_{n=0}^\infty b_nz^{n+mr_1}$$ and $$C_{\psi,\phi} e_{mr_2}(z)=\sqrt{\frac{\Gamma(mr_2+s)}{\Gamma(mr_2+1)\Gamma(s)}}\sum_{n=0}^\infty b_nz^{n+mr_2}.$$  
		Thus the matrix representation of $C_{\psi,\phi}$ on $P$ is
		
		$\begin{bmatrix}
			b_0&0&0\\
			\sqrt{\frac{\Gamma(mr_1+1)\Gamma(s)}{\Gamma(mr_1+s)}}b_{mr_1}&b_0&0\\
			\sqrt{\frac{\Gamma(mr_2+1)\Gamma(s)}{\Gamma(mr_2+s)}}b_{mr_2}&\sqrt{\frac{\Gamma(mr_1+s)\Gamma(mr_2+1)}{\Gamma(mr_1+1)\Gamma(mr_2+s)}}b_{m(r_2-r_1)} &b_0
		\end{bmatrix}.$
		
		Let $b_{m r_1}b_{mr_2}b_{m(r_1-r_2)}=0$ while at least one of the coefficients $b_{m r_1},b_{mr_2}, b_{m(r_1-r_2)} $ is non zero. Then it follows from \cite[Th. 4.1]{KRS_LAA_97} that the numerical range of the compression of $C_{\psi,\phi}$ to $P$ is the circular disc centered at $b_0$ with radius $\frac12 \sqrt{\frac{\Gamma(mr_1+1)\Gamma(s)}{\Gamma(mr_1+s)}|b_{mr_1}|^2+\frac{\Gamma(mr_2+1)\Gamma(s)}{\Gamma(mr_2+s)}|b_{mr_2}|^2+\frac{\Gamma(mr_1+s)\Gamma(mr_2+1)}{\Gamma(mr_1+1)\Gamma(mr_2+s)}|b_{m(r_2-r_1)}|^2}.$ Since the numerical range of compression is contained in the numerical range of the operator, the desired result follows.
	\end{proof}

	Moreover, we provide proofs of Theorems~\ref{Th_ss4} and~\ref{Th_ss5}, which demonstrate that, under appropriate assumptions, the numerical range $W(C_{\psi,\phi}; \mathcal D_s)$ contains an elliptical region.
	
	\begin{proof}[Proof of Theorem \ref{Th_ss4}]
		Suppose $P=\text{span}~\{e_0, e_{mr+k}\}.$ Then 
		$$C_{\psi,\phi} e_0(z)=\sum_{n=0}^\infty b_nz^n$$ and
		$$C_{\psi,\phi} e_{mr+k}(z)=e^{i \frac {2\pi(mr+k)}{m}}\sqrt{\frac{\Gamma(mr+k+s)}{\Gamma(mr+k+1)\Gamma(s)}}\sum_{n=0}^\infty b_nz^{mr+k+n} .$$
		Then the matrix representation of $C_{\psi,\phi}$ on $P$ is
		$\begin{bmatrix}
			b_0&0\\
			\sqrt{\frac{\Gamma(mr+k+1)\Gamma(s)}{\Gamma(mr+k+s)}}b_{mr+k}&b_0e^{i \frac {2\pi(mr+k)}{m}}
		\end{bmatrix}.$
		As $k\in(0,m),$ $e^{i \frac {2\pi(mr+k)}{m}}\neq1.$ Thus, the numerical range of the compression of $C_{\psi,\phi}$ to $P$ is an elliptical disc with foci at the points $b_0$, $b_0e^{i \frac {2\pi(mr+k)}{m}}$ with major axis of length $\sqrt{\frac{\Gamma(mr+k+1)\Gamma(s)}{\Gamma(mr+k+s)}|b_{mr+k}|^2+|b_0|^2|1-e^{i \frac {2\pi(mr+k)}{m}}|^2}$ and minor axis of length $\sqrt{\frac{\Gamma(mr+k+1)\Gamma(s)}{\Gamma(mr+k+s)}}|b_{mr+k}|.$ The desired conclusion now follows from the inclusion of the numerical range of a compression in the numerical range of the operator.
	\end{proof}

	\begin{proof}[Proof of Theorem \ref{Th_ss5}]
		Suppose $P=\text{span}~\{e_p, e_{p+q}\}.$ Then 
		$$C_{\psi,\phi} e_p(z)=e^{i2\pi p\phi}\sqrt{\frac{\Gamma(p+s)}{\Gamma(p+1)\Gamma(s)}}\sum_{n=0}^\infty b_nz^{p+n}$$ and
		$$C_{\psi,\phi} e_{p+q}(z)=e^{i2\pi (p+q)\phi}\sqrt{\frac{\Gamma(p+q+s)}{\Gamma(p+q+1)\Gamma(s)}}\sum_{n=0}^\infty b_nz^{p+q+n}.$$
		Then the matrix representation of $C_{\psi,\phi}$ on $P$ is
		$\begin{bmatrix}
			b_0e^{i2\pi p\phi}&0\\
			e^{i2\pi p\phi}\sqrt{\frac{\Gamma(p+q+1)\Gamma(p+s)}{\Gamma(p+q+s)\Gamma(p+1)}}b_{q}&b_0e^{i 2\pi(p+q)\phi}
		\end{bmatrix}.$
		Then the numerical range of the compression of $C_{\psi,\phi}$ to $P$ is an elliptical disc with foci at the points $ b_0e^{i2\pi p\phi}$, $b_0e^{i 2\pi(p+q)\phi}$ with major axis of length $$\sqrt{|b_0|^2|e^{i2\pi p\phi}-e^{i 2\pi(p+q)\phi}|^2+\frac{\Gamma(p+q+1)\Gamma(p+s)}{\Gamma(p+q+s)\Gamma(p+1)}|b_q|^2}$$ and minor axis of length $\sqrt{\frac{\Gamma(p+q+1)\Gamma(p+s)}{\Gamma(p+q+s)\Gamma(p+1)}}|b_{q}|$ . Hence, by the inclusion property of numerical ranges under compression, the proof is complete.
	\end{proof}

	\section{Weyl-type weighted composition operators on $\mathcal D_s$}\label{sct4}
	Let $\phi_{\gamma, \alpha}$ be the automorphism of the unit disc $\mathbb D$ defined by $\phi_{\gamma, \alpha}(z)=\alpha\frac{z-\gamma}{1-\bar{\gamma}z},$ where $\alpha\in\mathbb T$ and $\gamma\in\mathbb D.$ 
	We introduce the Weyl-type weighted composition operators $C_{\hat k_\gamma^{s}, \phi_{\gamma, \alpha}}$ on $\mathcal D_s$ defined as $$C_{\hat k_\gamma^{s}, \phi_{\gamma, \alpha}}(f)=\hat k_\gamma^{s}f\circ\phi_{\gamma, \alpha}\,\,\,\,\forall f\in \mathcal D_s.$$
	It follows from \cite[Th. 1.1]{Sen1} that $C_{\hat k_\gamma^{s}, \phi_{\gamma, \alpha}}$ is bounded on $\mathcal D_s$ and moreover, is invertible.

	The Berezin transform of $C_{\hat k_\gamma^{s}, \phi_{\gamma, \alpha}}$ on $\mathcal D_s$ is given by 
	
	\begin{eqnarray}
		\widetilde{C_{\hat k_\gamma^{s}, \phi_{\gamma, \alpha}}}(z)=\langle C_{\hat k_\gamma^{s}, \phi_{\gamma, \alpha}}\hat k_{z}^{s}, \hat k_{z}^{s} \rangle_{\mathcal D_s}
		&=&\frac{(1-|\gamma|^2)^{\frac{s}{2}}(1-|z|^2)^{s}}{(1-\bar{\gamma}z)^{s}(1-\bar{z}\phi_{\gamma, \alpha}(z))^{s}}\nonumber\\
		&=&\left(\frac{(1-|\gamma|^2)(1-|z|^2)^2}{\left((1-\alpha|z|^2)-(z\bar{\gamma}-\alpha\bar{z}\gamma)\right)^2}\right)^{\frac{s}{2}}.
	\end{eqnarray}

	Next, we provide the proof of Theorem~\ref{Th_sss_1}, which shows that although 
	$0$ lies in the closure of the Berezin range of $C_{\hat k_\gamma^{s}, \phi_{\gamma, \alpha}}$, it is not contained in the Berezin range itself.
	
	\begin{proof}[Proof of the Theorem \ref{Th_sss_1}]
		For any $z\in \mathbb D,$ we have 
		\begin{eqnarray}\label{eq_422}
			\widetilde{C_{\hat k_\gamma^{s}, \phi_{\gamma, \alpha}}}(z)
			&&=\frac{(1-|\gamma|^2)^{\frac{s}{2}}(1-|z|^2)^{s}}{(1-\bar{\gamma}z)^{s}(1-\bar{z}\phi_{\gamma, \alpha}(z))^{s}}.
		\end{eqnarray}

		Since $\phi_{\gamma, \alpha}$ is not an identity map on $\mathbb{D}$, there exists $z_0 \in \mathbb{T}$ such that the radial limit of $\phi_{\gamma, \alpha}$ exists at $z_0$ with $\phi_{\gamma, \alpha}^\ast(z_0) \neq z_0.$ Therefore, from \eqref{eq_422} we get $\widetilde{C_{\hat k_\gamma^{s}, \phi_{\gamma, \alpha}}}(z) \to 0$ when  $z$ radially approaches to the point $z_0.$ Hence, we obtain $0 \in \overline{\textbf{Ber}(C_{\hat k_\gamma^{s}, \phi_{\gamma, \alpha}}; \mathcal D_s)}.$
		
		Again, from \eqref{eq_422} we have
		\begin{align*}
			\left|\widetilde{C_{\hat k_\gamma^{s}, \phi_{\gamma, \alpha}}}(z)\right|&=\frac{(1-|\gamma|^2)^{\frac{s}{2}}(1-|z|^2)^{s}}{|1-\bar{\gamma}z|^{s}|1-\bar{z}\phi_{\gamma, \alpha}(z)|^{s}}\\
			&\geq \frac{(1-|\gamma|^2)^{\frac{s}{2}}(1-|z|^2)^{s}}{4^{{s}}}>0\,\,\text{for all $z \in \mathbb{D}$}.
		\end{align*}
		Therefore, we obtain $\widetilde{C_{\hat k_\gamma^{s}, \phi_{\gamma, \alpha}}}(z) \neq 0$ for all $z \in \mathbb{D},$ as desired.
	\end{proof}

	We now present the proof of Theorem~\ref{Th_sss_2}, which identifies the Berezin range of the operator $C_{\hat k_\gamma^{s}, \phi_{\gamma, \alpha}}$ on $\mathcal D_s$ for $\alpha=-1$.
	
	\begin{proof}[Proof of the Theorem \ref{Th_sss_2}]
		For $\alpha=-1,$ we have
		$$\langle C_{\hat k_\gamma^{s}, \phi_{\gamma, -1}}\hat k_{z}^{s}, \hat k_{z}^{s} \rangle_{\mathcal D_s}=\left(\frac{(1-|\gamma|^2)(1-|z|^2)^2}{\left((1+|z|^2)-(z\bar{\gamma}+\bar{z}\gamma)\right)^2}\right)^{\frac{s}{2}}.$$
		The equation $\phi_{\gamma, -1}(z)=z$ has a unique solution in $\mathbb D, $ $z(\gamma)=\frac{1-\sqrt{1-|\gamma|^2}}{\bar{\gamma}}.$ Thus, by simple computation we have 
		$$\widetilde{ C_{\hat k_\gamma^{s}, \phi_{\gamma, -1}}}(z(\gamma))=1.$$
		Since $\widetilde{ C_{\hat k_\gamma^{s}, \phi_{\gamma, -1}}}$ is a continuous function, it follows that the range of $\widetilde{ C_{\hat k_\gamma^{s}, \phi_{\gamma, -1}}}$ is a connected subset of $\mathbb R$ that contains 1 and $\widetilde{ C_{\hat k_\gamma^{s}, \phi_{\gamma, -1}}}(z)\to 0$ as $|z|\to1.$ 
		Therefore, $\textbf{Ber}\left(C_{\hat k_\gamma^{s}, \phi_{\gamma,-1}}; \mathcal D_s\right)=(0,1]$.
	\end{proof}

	We next provide the proof of Theorem~\ref{Th_sss_3}, which computes the Berezin number of the operator $C_{\hat k_\gamma^{s}, \phi_{\gamma, \alpha}}$ on $\mathcal D_s$ corresponding to $\alpha=1.$

	\begin{proof}[Proof of the Theorem \ref{Th_sss_3}]
		We have
		$$\langle C_{\hat k_\gamma^{s}, \phi_{\gamma, 1}}\hat k_{z}^{s}, \hat k_{z}^{s} \rangle_{\mathcal D_s}=\left(\frac{(1-|\gamma|^2)(1-|z|^2)^2}{\left((1-|z|^2)-(z\bar{\gamma}-\bar{z}\gamma)\right)^2}\right)^{\frac{s}{2}}.$$
		Since $i(z\bar{\gamma}-\bar{z}\gamma)$ is real, we get
		$$|(1-|z|^2)-(z\bar{\gamma}-\bar{z}\gamma)|^2=(1-|z|^2)^2+|z\bar{\gamma}-\bar{z}\gamma|^2.$$
		So, 
		\begin{eqnarray}\label{eq_sss1}
			|\langle C_{\hat k_\gamma^{s}, \phi_{\gamma, 1}}\hat k_{z}^{s}, \hat k_{z}^{s} \rangle_{\mathcal D_s}|=\left(\frac{(1-|\gamma|^2)(1-|z|^2)^2}{(1-|z|^2)^2+|z\bar{\gamma}-\bar{z}\gamma|^2}\right)^{\frac{s}{2}}\leq (1-|\gamma|^2)^{\frac{s}{2}}. 
		\end{eqnarray}
		
		Since $\widetilde{ C_{\hat k_\gamma^{s}, \phi_{\gamma, 1}}}(0)=(1-|\gamma|^2)^{\frac{s}{2}},$ it follows that $\textbf{ber}(C_{\hat k_\gamma^{s}, \phi_{\gamma, 1}}; \mathcal D_s)=(1-|\gamma|^2)^{\frac{s}{2}}.$
	\end{proof}
	
	We conclude this section by studying the Berezin radius of the sum of two Weyl-type weighted composition operators. For this class of operators, we show that a reverse power inequality for the Berezin radius holds.
	
	\begin{proposition}\label{propo_s111}
		Let $\mathcal X_\gamma=C_{\hat k_\gamma^{s}, \phi_{\gamma, 1}}+C_{\hat k^s_{-\gamma}, \phi_{-\gamma, 1}}.$ Then $\textbf{ber}(\mathcal X_\gamma; \mathcal D_s)=2(1-|\gamma|^2)^{\frac s2}$ and  $\textbf{ber}(\mathcal X_\gamma^2; \mathcal D_s)=2\left(\left(\frac{1-|\gamma|^2}{1+|\gamma|^2}\right)^s+1\right).$ Thus, for all $\gamma$ with $0<|\gamma|<1,$ $\textbf{ber}(\mathcal X_\gamma^2; \mathcal D_s)>\textbf{ber}^2(\mathcal X_\gamma; \mathcal D_s).$
	\end{proposition}
	
	\begin{proof}
		It follows from \eqref{eq_sss1} that,
		\begin{eqnarray*}
			|\langle \mathcal X_\gamma \hat k_{z}^{s},\hat k_{z}^{s} \rangle_{\mathcal D_s}|\leq |\langle C_{\hat k_\gamma^{s}, \phi_{\gamma, 1}} \hat k_{z}^{s},\hat k_{z}^{s} \rangle_{\mathcal D_s}| +|\langle C_{\hat k^s_{-\gamma}, \phi_{-\gamma, 1}}\hat k_{z}^{s},\hat k_{z}^{s} \rangle_{\mathcal D_s}|
			\leq 2(1-|\gamma|^2)^{\frac s2}.
		\end{eqnarray*}
		As $\widetilde{\mathcal X_\gamma (0)}=2(1-|\gamma|^2)^{\frac s2},$ thus, $\textbf{ber}(\mathcal X_\gamma; \mathcal D_s)=2(1-|\gamma|^2)^{\frac s2}.$\\
		Since  $C_{\hat k_\gamma^{s}, \phi_{\gamma, 1}}C_{\hat k^s_{-\gamma}, \phi_{-\gamma, 1}}=C_{\hat k^s_{-\gamma}, \phi_{-\gamma, 1}}C_{\hat k_\gamma^{s}, \phi_{\gamma, 1}}=I,$ we have $$\mathcal X_\gamma^2=C_{\hat k_\gamma^{s}, \phi_{\gamma, 1}}^2+C_{\hat k^s_{-\gamma}, \phi_{-\gamma, 1}}^2+2I.$$ Note that, $C_{\hat k_\gamma^{s}, \phi_{\gamma, 1}}^2=C_{\hat k_{\frac{2\gamma}{1+|\gamma|^2}}^{s}, \phi_{\frac{2\gamma}{1+|\gamma|^2}, 1}}.$ 
		Hence,
		\begin{eqnarray*}
			|\langle \mathcal X_\gamma^2 \hat k_{z}^{s},\hat k_{z}^{s} \rangle_{\mathcal D_s}|&\leq&\left |\left\langle C_{\hat k_{\frac{2\gamma}{1+|\gamma|^2}}^{s}, \phi_{\frac{2\gamma}{1+|\gamma|^2}, 1}} \hat k_{z}^{s},\hat k_{z}^{s} \right\rangle_{\mathcal D_s}\right| +\left|\left\langle C_{\hat k_{\frac{-2\gamma}{1+|\gamma|^2}}^{s}, \phi_{\frac{-2\gamma}{1+|\gamma|^2}, 1}}\hat k_{z}^{s},\hat k_{z}^{s} \right\rangle_{\mathcal D_s}\right|+2\\
			&\leq&2\left(1-\frac{4|\gamma|^2}{(1+|\gamma|^2)^2}\right)^{\frac s2}+2=2\left(\left(\frac{1-|\gamma|^2}{1+|\gamma|^2}\right)^s+1\right).   
		\end{eqnarray*}
		As $\widetilde{\mathcal X_\gamma^2 (0)}=2\left(\left(\frac{1-|\gamma|^2}{1+|\gamma|^2}\right)^s+1\right),$ thus, $\textbf{ber}(\mathcal X_\gamma^2; \mathcal D_s)=2\left(\left(\frac{1-|\gamma|^2}{1+|\gamma|^2}\right)^s+1\right).$\\
		Since $0<|\gamma|<1,$ we have $(1-|\gamma|^2)^s<1$ and $\frac{1}{(1+|\gamma|^2)^s}<1$. Consequently, we obtain $$\left(\left(\frac{1-|\gamma|^2}{1+|\gamma|^2}\right)^s+1\right)>2(1-|\gamma|^2)^{s}.$$ 
		This completes the proof.
	\end{proof}
	
	\begin{remark}
		In \cite{GA_CAOT_2021}, authors exhibit a class of positive bounded linear operators $T$ on complex Hilbert space $\mathcal H$ that satisfiy the reverse power inequality for the Berezin radius, namely,
		$$\textbf{ber}(T^p; \mathcal H)\geq\textbf{ber}^p(T; \mathcal H),~p\in\mathbb N.$$ 
		In Proposition~\ref{propo_s111}, we obtain a new class of operators $\mathcal X_\gamma$ on $\mathcal D_s$ for which the sharper reverse power inequality holds for $p=2$.
	\end{remark}

	\section{Convexity of Berezin range of composition operators on $\mathcal D_s$} \label{sct5}

	For the composition operator $C_{\phi_{\gamma,\xi}}$ on $\mathcal D_s,$ the Berezin transform is given by
	\begin{align}\label{e1}
		\widetilde{C_{\phi_{\gamma,\xi}}}(z)
		=\frac{1}{\|k^{s}_z\|_{\mathcal D_s}^2}\langle C_{\phi_{\gamma,\xi}}{k}^{s}_z, {k}^{s}_z \rangle_{\mathcal D_s}
		=\frac{1}{\|k^{s}_z\|_{\mathcal D_s}^2}{k}^{s}_z(\phi_{\gamma,\xi}(z))
		=\left(\frac{1-|z|^2}{1-\bar{z}\phi_{\gamma,\xi}(z)}\right)^{s}.
	\end{align}

	For $\xi\in\overline{\mathbb D},$ define $\phi_{0,\xi}(z)=\xi z.$  
	\begin{equation*}
		\widetilde{C_{{\phi_{0,\xi}}}}(z)
		=\left(\frac{1-|z|^2}{1-\xi |z|^2}\right)^{s},
	\end{equation*}
	where $0<s<1$.

	We begin with the following lemma, which is required for the proof of Theorem~\ref{Th_ssss1}.

	\begin{lemma}\label{lm_1}
		Suppose that ${\phi_{0,\xi}}(z)=\xi z,$ where $\xi\in\overline{\mathbb D}$ and $z\in\mathbb D$. Then 
		\begin{align*}
			&(i)~~\textbf{Ber}(C_{{\phi_{0,\xi}}}; \mathcal D_s)~~\text{is singleton if and only if}~~\xi=1,\\
			&(ii)~~\textbf{Ber}(C_{{\phi_{0,\xi}}};\mathcal D_s)\subset \mathbb R~~\text{if and only if}~~\xi\in\mathbb R.
		\end{align*}
	\end{lemma}

	\begin{proof}
		$(i)$ Since the sufficiency follows directly, we address only the necessity.
		Let $z=\rho e^{i\psi}$ with $\rho\in[0, 1)$ and $\psi\in[0,2\pi).$ Then 
		\begin{equation*}
			\widetilde{C_{{\phi_{0,\xi}}}}(\rho e^{i\psi})=\left(\frac{1-\rho^2}{1-\xi \rho^2}\right)^{s}.
		\end{equation*}
		Suppose that $\textbf{Ber}(C_{{\phi_{0,\xi}}}; \mathcal D_s)$ is singleton. As $\widetilde{C_{{\phi_{0,\xi}}}}(0)=1,$ so $\textbf{Ber}(C_{{\phi_{0,\xi}}};\mathcal D_s)=\{1\}.$
		Thus, $\left(\frac{1-\rho^2}{1-\xi \rho^2}\right)^{s}=1$ for all $\rho\in[0, 1).$
		This implies that 
		\begin{equation}\label{eq2}
			\left|\frac{1-\rho^2}{1-\xi \rho^2}\right|=1\implies 1-\rho^2=|1-\xi \rho^2|.
		\end{equation}
		Thus, $1-\rho^2\geq 1-|\xi| \rho^2$ and so $|\xi|\geq 1.$ Since $\xi \in\overline{\mathbb D},$ we have $|\xi|=1.$
		Substituting $\xi=\cos \theta+i \sin \theta$ into \eqref{eq2}, we obtain 
		$$1-\rho^2=|(1-\rho^2\cos \theta)-i~\rho^2 \sin \theta|.$$
		This implies that $\cos\theta=1.$ Thus, $\sin \theta=0.$\\
		$(ii)$ Let $\textbf{Ber}(C_{{\phi_{0,\xi}}};\mathcal D_s)\subseteq \mathbb R$. Then 
		$\frac{1-\rho^2}{1-\xi \rho^2}\in\mathbb R$ for all $\rho\in[0, 1).$ Let $\frac{1-\rho^2}{1-\xi \rho^2}=m_\rho,$ where $m_\rho\in\mathbb R.$ Clearly, $m_\rho\neq 0$ for all $\rho\in[0, 1).$ Then
		$\xi=\frac{1}{\rho^2}-\frac{1-\rho^2}{m_\rho \rho^2}\in\mathbb R.$ Thus, $\xi\in\mathbb R.$ The converse is immediate.
	\end{proof}

	We are now in a position to establish Theorem~\ref{Th_ssss1}, thereby determining the values of $\xi$ that guarantee the convexity of the Berezin range of $C_{\phi_{0,\xi}}.$
	
	\begin{proof}[Proof of Theorem~\ref{Th_ssss1}]
		Consider $z=\rho e^{i\psi},$ where $\rho\in[0, 1)$ and $\psi\in[0,2\pi).$ Then 
		\begin{equation}\label{eq_lm_1}
			\widetilde{C_{{\phi_{0,\xi}}}}(\rho e^{i\psi})=\left(\frac{1-\rho^2}{1-\xi \rho^2}\right)^{s}.
		\end{equation}
		For $\xi=1,$ Lemma \ref{lm_1} $(i)$ yields that $\textbf{Ber}(C_{{\phi_{0,\xi}}};\mathcal D_s)$ is a single point and thus convex.
		For $\xi\in [-1,1),$ we have $\textbf{Ber}(C_{{\phi_{0,\xi}}};\mathcal D_s)=(0, 1],$ which is also convex.\\
		Conversely, assume that $\textbf{Ber}(C_{{\phi_{0,\xi}}};\mathcal D_s)$ is convex. Equation \eqref{eq_lm_1} shows that $\widetilde{C_{{\phi_{0,\xi}}}}(\rho e^{i\psi})$ does not depend on $\psi.$ Hence, $\textbf{Ber}(C_{{\phi_{0,\xi}}};\mathcal D_s)$ forms a path in $\mathbb C.$ Thus, convexity ensures that $\textbf{Ber}(C_{{\phi_{0,\xi}}};\mathcal D_s)$ reduces to either a line segment or a single point. If $\textbf{Ber}(C_{{\phi_{0,\xi}}};\mathcal D_s)$ reduces to a point then by Lemma \ref{lm_1} $(i)$ we must have $\xi=1$. Let $\textbf{Ber}(C_{{\phi_{0,\xi}}};\mathcal D_s)$ be a line segment. We have $\widetilde{C_{{\phi_{0,\xi}}}}(0)=1$ and $\lim\limits_{\rho\to 1^-}\widetilde{C_{{\phi_{0,\xi}}}}(\rho e ^{i\psi})=0.$ It follows that $\textbf{Ber}(C_{{\phi_{0,\xi}}};\mathcal D_s)$ is a line segment passing through the point 1 and approaching the origin and hence $\textbf{Ber}(C_{{\phi_{0,\xi}}};\mathcal D_s)\subset\mathbb R.$ By Lemma \ref{lm_1} $(ii)$, we conclude that 
		$\xi\in\mathbb R$. Thus, $\xi\in[-1, 1].$ 
	\end{proof}
	
	Our next goal is to analyze the convexity of the Berezin range of $C_{\phi_{\gamma, 1}},$ where $\phi_{\gamma, 1}$ is a Blaschke factor.

	Let $\phi_{\gamma, 1}(z)=\frac{z-\gamma}{1-\bar{\gamma}z}$ where $\gamma, z \in\mathbb D.$ Then 
	\begin{equation}\label{eq3}
		\widetilde{C_{\phi_{\gamma, 1}}}(z)=\left(\frac{1-|z|^2}{1-\bar{z}\phi_{\gamma, 1}(z)}\right)^{s}= \left(\frac{(1-|z|^2)(1-\bar{\gamma}z)}{1-\bar{\gamma}z-|z|^2+\bar{z}\gamma}\right)^{s},
	\end{equation}
	where $0<s<1$.

	We now present the following lemmas, which form the foundation for the proof of Theorem~\ref{Th_ssss2}.
	
	\begin{lemma}\label{lm3}
		Suppose $C_{\phi_{\gamma, 1}}\in\mathcal B(\mathcal D_s)$ with $\phi_{\gamma, 1}(z)=\frac{z-\gamma}{1-\bar{\gamma}z},$ where $ \gamma, z \in\mathbb D.$ Then
		$Im\{\widetilde{C_{\phi_{\gamma, 1}}}(z)\}=0$ if and only if $Im\{\bar{\gamma}z\}=0$.
	\end{lemma}

	\begin{proof}
		
		It follows from \eqref{eq3} that
		
		\small{\begin{eqnarray}\label{eq_p1}
				&&\widetilde{C_{\phi_{\gamma, 1}}}(z)\nonumber\\
				&&=\left(\frac{(1-|z|^2)(1-\bar{\gamma}z)(1-|z|^2-2i Im\{\bar{z}\gamma\})}{|1-|z|^2+2i Im\{\bar{z}\gamma\}|^2}\right)^{s}\nonumber\\
				&&=\left(\frac{1-|z|^2}{|1-|z|^2+2iIm \{\gamma \bar{z}\}|^2}\right)^s\nonumber\\
				&&\left((1-|z|^2)(1-Re\{\bar{\gamma}z\})+2Im ^2\{\bar{\gamma}z\}
				+i Im\{\bar{\gamma}z\}(1+|z|^2-2Re\{\bar{\gamma}z\})\right)^s.
		\end{eqnarray}}
		
		Let
		\begin{eqnarray*}
			(1-|z|^2)(1-Re\{\bar{\gamma}z\})+2Im ^2\{\bar{\gamma}z\}=\rho \cos \theta
		\end{eqnarray*}
		and 
		\begin{eqnarray*}
			Im\{\bar{\gamma}z\}(1+|z|^2-2Re\{\bar{\gamma}z\})=\rho \sin\theta,
		\end{eqnarray*}
		where $\rho=\left(((1-|z|^2)(1-Re\{\bar{\gamma}z\})+2Im^2\{\bar{\gamma}z\})^2+Im^2\{\bar{\gamma}z\}(1+|z|^2-2Re\{\bar{\gamma z}\})^2\right)^{\frac{1}{2}}$ and $-\frac{\pi}{2}<\theta<\frac{\pi}{2}.$ It follows from \eqref{eq_p1} that
		\begin{eqnarray*}
			\widetilde{C_{\phi_{\gamma, 1}}}(z)
			&=&\left(\frac{\rho(1-|z|^2)}{|1-|z|^2+2iIm \{\gamma \bar{z}\}|^2}\right)^s(\cos s\theta+i\sin s\theta).
		\end{eqnarray*}
		
		Thus, 
		\begin{eqnarray}
			&&Re\{\widetilde{C_{\phi_{\gamma, 1}}}(z)\}\nonumber\\
			&&=\left(\frac{(1-|z|^2)}{|1-|z|^2+2iIm \{\gamma \bar{z}\}|^2}\right)^s\nonumber\\
			&&\Big(((1-|z|^2)(1-Re\{\bar{\gamma}z\})+2Im^2\{\bar{\gamma}z\})^2+Im^2\{\bar{\gamma}z\}(1+|z|^2-2Re\{\bar{\gamma z}\})^2\Big)^{\frac{s}{2}}\nonumber\\
			&&\cos \left(s\tan^{-1}\left(\frac{Im\{\bar{\gamma}z\}(1+|z|^2-2Re\{\bar{\gamma z}\})}{(1-|z|^2)(1-Re\{\bar{\gamma}z\})+2Im^2\{\bar{\gamma}z\}}\right)\right)
		\end{eqnarray}
		and 
		\begin{eqnarray}\label{eqp_2}
			&&Im\{\widetilde{C_{\phi_{\gamma, 1}}}(z)\}\nonumber\\
			&&=\left(\frac{(1-|z|^2)}{|1-|z|^2+2iIm \{\gamma \bar{z}\}|^2}\right)^s\nonumber\\
			&&\Big(((1-|z|^2)(1-Re\{\bar{\gamma}z\})+2Im^2\{\bar{\gamma}z\})^2+Im^2\{\bar{\gamma}z\}(1+|z|^2-2Re\{\bar{\gamma z}\})^2\Big)^{\frac{s}{2}}\nonumber\\
			&& \sin \left(s\tan^{-1}\left(\frac{Im\{\bar{\gamma}z\}(1+|z|^2-2Re\{\bar{\gamma z}\})}{(1-|z|^2)(1-Re\{\bar{\gamma}z\})+2Im^2\{\bar{\gamma}z\}}\right)\right).
		\end{eqnarray}
		It follows from \eqref{eqp_2} that if $Im\{\bar{\gamma}z\}=0$ then $Im\{\widetilde{C_{\phi_{\gamma, 1}}}(z)\}=0.$ Conversely, assume that $Im\{\widetilde{C_{\phi_{\gamma, 1}}}(z)\}=0.$ As
		$\frac{(1-|z|^2)}{|1-|z|^2+2iIm \{\gamma \bar{z}\}|^2}>0, (1-|z|^2)(1-Re\{\bar{\gamma}z\})+2Im^2\{\bar{\gamma}z\}>0$ and $1+|z|^2-2Re\{\bar{\gamma z}\}>0,$ \eqref{eqp_2} implies that
		\begin{eqnarray}\label{eq_p4}
			\sin \left(s\tan^{-1}\left(\frac{Im\{\bar{\gamma}z\}(1+|z|^2-2Re\{\bar{\gamma z}\})}{(1-|z|^2)(1-Re\{\bar{\gamma}z\})+2Im^2\{\bar{\gamma}z\}}\right)\right)=0.
		\end{eqnarray}
		Since $(1-|z|^2)(1-Re\{\bar{\gamma}z\})+2Im^2\{\bar{\gamma}z\}>0,$ it follows that
		$$-\frac{\pi}{2}<\tan^{-1}\left(\frac{Im\{\bar{\gamma}z\}(1+|z|^2-2Re\{\bar{\gamma z}\})}{(1-|z|^2)(1-Re\{\bar{\gamma}z\})+2Im^2\{\bar{\gamma}z\}}\right)<\frac{\pi}{2}.$$ \
		As $0<s< 1$, we have 
		\begin{eqnarray}\label{eqp_5}
			-\frac{\pi}{2}<s\tan^{-1}\left(\frac{Im\{\bar{\gamma}z\}(1+|z|^2-2Re\{\bar{\gamma z}\})}{(1-|z|^2)(1-Re\{\bar{\gamma}z\})+2Im^2\{\bar{\gamma}z\}}\right)<\frac{\pi}{2}.
		\end{eqnarray}
		Combining \eqref{eq_p4} and \eqref{eqp_5}, we have
		$\tan^{-1}\left(\frac{Im\{\bar{\gamma}z\}(1+|z|^2-2Re\{\bar{\gamma z}\})}{(1-|z|^2)(1-Re\{\bar{\gamma}z\})+2Im^2\{\bar{\gamma}z\}}\right)=0.$ This implies that 
		$Im\{\bar{\gamma}z\}=0$.
	\end{proof}

	In light of the preceding result, we provide the proof of Theorem~\ref{Th_ssss2}, which characterizes the convexity of the Berezin range of $C_{\phi_{\gamma, 1}}$ on $\mathcal D_s$.

	\begin{proof}[Proof of Theorem~\ref{Th_ssss2}]
		For $\gamma=0$, $\textbf{Ber}(C_{\phi_{\gamma, 1}};\mathcal D_s)=\{1\}$ is convex. Conversely, let $\textbf{Ber}(C_{\phi_{\gamma, 1}}; \mathcal D_s)$ be convex. Now we show that for any $z \in \mathbb D,$ $Re\{\widetilde{C_{\phi_{\gamma, 1}}}(z)\}\in\textbf{Ber}(C_{\phi_{\gamma, 1}};\mathcal D_s).$ Let $z=\rho e^{i\psi}$ and $\gamma=r e^{i\theta},$ for $\rho, r \in [0,1)$ and $\psi, \theta \in [0,2\pi).$ A straightforward computation yields 
		$$\widetilde{C_{\phi_{\gamma, 1}}}(\rho e^{i(2\theta-\psi)})=\overline{\widetilde{C_{\phi_{\gamma, 1}}}(z)} \in \textbf{Ber}(C_{\phi_{\gamma, 1}};\mathcal D_s).$$ 
		Since $\textbf{Ber}(C_{\phi_{\gamma, 1}};\mathcal D_s)$ is convex, so we have 
		$$Re\{\widetilde{C_{\phi_{\gamma, 1}}}(z)\}=\frac 12\left(\widetilde{C_{\phi_{\gamma, 1}}}(z)+\overline{\widetilde{C_{\phi_{\gamma, 1}}}(z)}\right)\in \textbf{Ber}(C_{\phi_{\gamma, 1}};\mathcal D_s).$$ It follows that for every
		$z\in\mathbb D,$ there exists $w\in\mathbb D$ such that $\widetilde{C_{\phi_{\gamma, 1}}}(w)=Re\{\widetilde{C_{\phi_{\gamma, 1}}}(z)\}.$ Then $Im\{\widetilde{C_{\phi_{\gamma, 1}}}(w)\}=0.$ From Lemma \ref{lm3}, $Im\{\bar{\gamma}w\}=0$ and so $w=t\gamma$ for some $t\in\left(-\frac{1}{|\gamma|}, \frac{1}{|\gamma|}\right)$. Hence 
		\begin{align*}
			&\widetilde{C_{\phi_{\gamma, 1}}}(w)\\&=Re\{\widetilde{C_{\phi_{\gamma, 1}}}(t\gamma)\}\\
			&=\left(\frac{1-|t\gamma|^2}{|1-|t\gamma|^2+2iIm \{\gamma \bar{t\gamma}\}|^2}\right)^{s}\\
			&\Big(((1-|t\gamma|^2)(1-Re\{\bar{\gamma}t\gamma\})+2Im^2\{\bar{\gamma}t\gamma\})^2+Im^2\{\bar{\gamma}t\gamma\}(1+|t\gamma|^2-2Re\{\bar{\gamma t\gamma}\})^2\Big)^{\frac{s}{2}}\\
			&\cos \left(s\tan^{-1}\left(\frac{Im\{\bar{\gamma}t\gamma\}(1+|t\gamma|^2-2Re\{\bar{\gamma t\gamma}\})}{(1-|t\gamma|^2)(1-Re\{\bar{\gamma}t\gamma\})+2Im^2\{\bar{\gamma}t\gamma\}}\right)\right)\\
			&=(1-t|\gamma|^2)^{s}.
		\end{align*}
		Therefore, we obtain
		$$\left\{\widetilde{C_{\phi_{\gamma, 1}}}(t\gamma):t\in\left(-\frac{1}{|\gamma|}, \frac{1}{|\gamma|}\right)\right\}=\left((1-|\gamma|)^{s}, (1+|\gamma|)^{s}\right).$$ 
		We have
		$$
		\lim\limits_{\rho\to 1^-}\widetilde{C_{\phi_{\gamma, 1}}}(\rho e^{i\psi})=
		\begin{cases}
			0&\text{for}~ \gamma\neq0\\
			1&\text{for}~ \gamma=0.
		\end{cases}
		$$
		It follows that when $\gamma\neq0,$ given $\epsilon$ with $0<\delta<(1-|\gamma|)^{s},$ there exists a point $z$ such that  $|Re\{\widetilde{C_{\phi_{\gamma, 1}}}(z)\}|<\delta.$ Now, for this $z$ there exists $w\in\mathbb D$ such that $|\widetilde{C_{\phi_{\gamma, 1}}}(w)|=|Re\{\widetilde{C_{\phi_{\gamma, 1}}}(z)\}|<\epsilon.$ This contradicts 
		$\widetilde{C_{\phi_{\gamma, 1}}}(w)\in((1-|\gamma|)^{s}, (1+|\gamma|)^{s}).$
		Therefore, $\gamma=0.$
	\end{proof}

	\textbf{Final remark:} We investigate the convexity of the Berezin range of unweighted composition operators on $\mathcal D_s$ and, in Section \ref{sct4}, analyze the Berezin range of Weyl-type weighted composition operators in particular cases. This naturally leads to the following open problem.
	\begin{question}
		Characterize the symbols $\phi$ and $\psi$ that ensure the convexity of the Berezin range of $C_{\psi,\phi}$ on $\mathcal D_s$.
	\end{question}

	\section*{Declarations}	
	\textit{Acknowledgements.} Miss Somdatta Barik would like to thank UGC, Govt. of India, for the financial support in the form of Senior Research Fellowship under the mentorship of Prof. Kallol Paul.
	Dr. Anirban Sen is supported by Czech Science Foundation (GA CR) grant no. 25-18042S.\\
	\textit{Author Contributions:} All authors contributed equally to this manuscript and approved the final version.\\
	\textit{Data Availability :} No data was used for the research described in this article.\\
	\textit{Conflict of interest:} The authors declare that they have no conflict of interest.\\

\end{document}